\documentclass[12pt]{article}

\usepackage{geometry, graphicx, setspace, float, color, enumitem, authblk, amssymb, amsmath, amsthm, mathtools, mathdots}
\usepackage[linktoc=all,hidelinks]{hyperref}

\usepackage{fancyvrb}
\usepackage{fvextra}
\fvset{breaklines=true}
\fvset{breakanywhere=true}

\newcommand{\bra}[1]{\langle#1\rangle}
\DeclareMathOperator{\supp}{supp}

\newtheorem{Thm}{Theorem}
\newtheorem{Lm}{Lemma}
\newtheorem{Cor}{Corollary}

\title{Principal Non-singularity of Fourier Matrices on $\mathbb Z_p \times \mathbb Z_q$ and $\mathbb Z_2^k \times \mathbb Z_q$}
\author{Weiqi Zhou\thanks{zwq@xzit.edu.cn}}
\affil{\small School of Mathematics and Statistics, Xuzhou University of Technology \\  {\footnotesize Lishui Road 2, Yunlong District, Xuzhou, Jiangsu Province, China 221111}}
\date{}							

\begin{document}
\maketitle
\begin{abstract}
Let $F_n$ be the $n\times n$ Fourier matrix on the cyclic group $\mathbb Z_n$, a renowned theorem of Chebotar\"ev asserts that all minors in $F_n$ for prime $n$ are non-zero. In this short note it is shown that (i) all principal minors in the Kronecker product $F_p\otimes F_q$ are non-vanishing (principal non-singularity) for distinct odd primes $p,q$ if $q$ is large enough and generates the multiplicative group $\mathbb Z_p^*$; (ii) the Fourier matrix on $\mathbb Z_2^k \times \mathbb Z_q$ is principally non-singular upon permutation (in particular, for $k=1$ the identity permutation suffices) for odd prime $q$ and $k=1,2,3$. The proof is just an exposition of existing techniques reorganized in a unified way. The result will have implications in combining Riesz bases of exponentials. \\

{\noindent 
}
\end{abstract}
 
\section{Introduction}
Let $A$ be a square matrix whose rows and columns are indexed by the same set, and $\mathcal I, \mathcal J$ subsets of the index set. Denote by $A[\mathcal I, \mathcal J]$ the submatrix obtained by taking row indices from $\mathcal I$ and column indices from $\mathcal J$ (If $\mathcal I, \mathcal J$ are singletons then it is the corresponding entry, i.e., $A[i,j]$ means the $i,j$-th entry). A \emph{minor} is the determinant of a submatrix satisfying $|\mathcal I|=|\mathcal J|$. If in addition $\mathcal I=\mathcal J$, then the submatrix is called \emph{principal} and its determinant is called a \emph{principal minor}. We shall say that a matrix is \emph{principally non-singular} if all its principal minors are non-zero, and \emph{principally non-singular upon permutation} if there is a column permutation so that all principal minors in the permuted matrix are non-zero.

Let $\mathbb Z_n=\{0,1,\ldots, n-1\}$ be the additive cyclic group of order $n$, and set $\zeta_n=e^{2\pi i/n}$. The \emph{Fourier matrix} on the finite Abelian group $\mathbb Z_{n_1}\times\ldots\times \mathbb Z_{n_d}$ is the matrix indexed by $\mathbb Z_{n_1}\times\ldots\times \mathbb Z_{n_d}$ with its $j,k$-th entry given by $\prod_{\ell=1}^d\zeta_{n_{\ell}}^{j_{\ell}k_{\ell}}$ where both $j=(j_1,\ldots,j_d)$ and $k=(k_1,\ldots,k_d)$ are in $\mathbb Z_{n_1}\times\ldots\times \mathbb Z_{n_d}$. It is typical to further attach the factor $n^{-1/2}$ with $n=n_1\ldots n_d$ to make the matrix unitary, but in this article it is more convenient not to include this scaling factor. When $d=1$, it is then the usual $n\times n$ DFT matrix scaled by $\sqrt n$ and we denote it by $F_n$. Therefore $F_n$ can also be viewed as an $n\times n$ Vandermonde matrix on $n$-th roots of unity.

A celebrated theorem of Chebotar\"ev (see e.g., \cite{stevenhagen1996}) states that if $p$ is a prime number, then all minors in $F_p$ are non-zero. For composite $n$ this is in general not true, for example if $n=4$ and $\mathcal I=\{0,2\}$, then all entries in the principal submatrix $F_4[\mathcal I, \mathcal I]$ are identically $1$. This theorem is also equivalent to the enhanced additive (compared to the usual productive uncertainty principle as in \cite{donoho1989, smith1990}) uncertainty principle (see \cite{tao2005}) that $|\supp(\vec x)|+|\supp(F_p\vec x)|\ge p+1$, where $\supp(\vec x)=\{k\in\mathbb Z_p: x_k\neq 0\}$ is the \emph{support} of $\vec x$.  There are also other additive improvements of the uncertainty inequality for composite $n$, see, e.g., \cite{meshulam2006}, but to translate uncertainty inequalities to principal non-singularity of Fourier matrices, an additive inequality on subvectors in the form of \cite[Proposition 2]{loukaki2025} will be needed.

There have been efforts on understanding whether permuted Fourier matrices can be principally non-singular, this is primarily motivated by classical questions in sampling theory such as: combining Riesz exponential bases on disjoint sets to form a new Riesz exponential basis on the union of these sets, or conversely how can we split a Riesz exponential basis on a set to get Riesz exponential bases on each individual component in a partition of the original set. In the introductory part of e.g., \cite{kozma2015, pfander2024} there are rich examples showing the complexity of this problem. A \emph{Riesz basis} is the image of an orthonormal basis under a bounded (away from both $\infty$ and $0$) linear operator. In the absence of an orthogonal exponential basis (e.g., triangles and disks admit no orthogonal Fourier bases \cite{fuglede1974}), a Riesz exponential basis would be a good choice for analyzing and reconstructing functions. 

Denote the multiplicative group modulo $n$ by $\mathbb Z_n^*$. In simpler settings such as when the underlying set is a finite collection of intervals with rational end points, then the possibility of combining or splitting Riesz exponential bases will rely critically on the principal non-singularity of $F_n$ upon permutations. The conjecture that $F_n$ is principally non-singular upon permutation (in particular it is conjectured that the identity permutation shall suffice if $n$ is square free) is thereof posed in the recent work \cite{cabrelli2025, caragea2025}. This is supported by numerical experiments up to $n=30$ and proved for $n=2p$ and $n=qr$ in \cite{loukaki2025} for odd primes $p,q,r$ with $q$ large enough and generating $\mathbb Z_r^*$. It is also known (see e.g., \cite{barnett2022}) that as $n$ grows the smallest non-zero principal minor in $n^{-1/2}F_n$ (scaled to be unitary) decay exponentially (e.g., the minimum principal minor is about $10^{-11}$ in for $n=30$), which makes them difficult to detect numerically. Principal non-singularity of a Fourier matrix on a finite Abelian group $G$ can also be interpreted as a form of uncertainty principle that vectors on $G$ and their Fourier transforms can not have disjoint supports.

Let $A\in\mathbb C^{m\times n}$ and $B\in\mathbb C^{m'\times n'}$, the Kronecker product $A\otimes B$ is an $mm'\times nn'$ matrix defined as
\begin{equation} \label{EqDefK}
A\otimes B=\begin{pmatrix}
A[0,0]B & \ldots & A[0, n-1]B \\
\vdots & \ddots & \vdots \\
A[m-1,0]B & \ldots & A[m-1, n-1]B
\end{pmatrix}.
\end{equation}

Since any $k\in\mathbb Z_{mn}$ can be uniquely written as $k=q_kn+r_k$ for some $q_k\in\mathbb Z_m$ and some $r_k\in\mathbb Z_n$, from \eqref{EqDefK} it is clear that $F_m\otimes F_n$ is the Fourier matrix on $\mathbb Z_m\times \mathbb Z_n$ with the index correspondence being that the $j,k$-th entry in $F_m\otimes F_n$ corresponds to the  $(q_j,q_k),(r_j,r_k)$-th entry in the Fourier matrix on $\mathbb Z_m\times \mathbb Z_n$. This also shows that $F_m\otimes F_n$ only differs from $F_n\otimes F_m$ by a simultaneous permutation. Moreover, if $m,n$ are co-primes, then $\mathbb Z_{mn}\cong \mathbb Z_m\times\mathbb Z_n$ implies that $F_m\otimes F_n$ also differs from $F_{mn}$ by row and column (not necessarily simultaneous) permutations in such case.

The purpose of this short note is to show in Corollary \ref{CorPQ} and Corollary \ref{Cor2Q} respectively that $F_p\otimes F_q$ is principally non-singular for distinct odd primes $p,q$ if $q$ is large enough and generates $Z_p^*$, and that $\underbrace{F_2 \otimes\ldots\otimes F_2}_{k \text{ terms}}\otimes F_q$ is principally non-singular upon permutation (in particular, for $k=1$ the identity permutation suffices) for odd prime $q$ and $k=1,2,3$. The proof is essentially the same as in \cite{frenkel2003, loukaki2025} with the method summarized and unified in Theorem \ref{ThmMain}. 
  
\section{Main Results}
The following lemma is originally stated in \cite[Lemma 2]{frenkel2003} for polynomials on finite fields of prime orders, it actually holds in many cases (e.g., it is true in $\mathbb C[x]$) since all we need for it to hold is just a polynomial ring with a well defined notion of non-vanishing formal derivatives that satisfies the product rule and allows factorization into linear factors, hence we state it for polynomials on integral domains for most generality (see e.g., \cite[Chapter III-6, Theorem 6.7-Theorem 6.10]{hungerford1974}).
\begin{Lm}\cite[Lemma 2]{frenkel2003} \label{LmF1}
Let $D$ be an integral domain of characteristic $q$, and $p(x)\in D[x]$ a polynomial. Suppose that $a\neq 0$ is a root of $p(x)$ with multiplicity $m$, and $h$ is the number of non-zero coefficients of $p(x)$, then $m<h$ holds if one of the following two conditions holds:
\begin{itemize}[leftmargin=*]
\item $q>0$ and the order of $p(x)$ is less than $q$;
\item $q=0$.
\end{itemize}
\end{Lm}

\begin{proof}
We may without loss of generality assume that the constant term of $p(x)$ is not $0$, otherwise we may simply divide the lowest term out. The proof is then by induction on the order of polynomials. For linear polynomials the statement trivially holds. Suppose it holds for all polynomials of order up to $k-1$, and consider now an order $k$ polynomial $p(x)$. Since $D$ is an integral domain, not only can we write $p(x)=(x-a)^mr(x)$ for some polynomial $r(x)$ with $r(a)\neq 0$, but also we can apply the product rule to compute its formal derivative to be $p'(x)=(x-a)^{m-1}\left(mr(x)+(x-a)r'(x)\right)$, and verify that the multiplicity of $a$ as a root of $p'(x)$ is $m-1$. Indeed, the term $m(x-a)^{m-1}r(x)$ is non-vanishing since either the characteristic is $0$ or $m\le k<q$. Since $p(x)$ has a non-zero constant term, the number of non-zero coefficient of $p'(x)$ is precisely $h-1$, thus by the induction assumption we get $m-1<h-1$, which leads to $m<h$ as desired.
\end{proof}

Remark: The arguments may appear to be trivial at first glance, but it is important to carry it out on integral domains, otherwise the factorization may not be well defined.

\begin{Thm} \label{ThmMain}
Let $m,n\in\mathbb N$, and suppose that $F$ is a matrix in which every entry is an integer power of $\zeta_m$, then $F\otimes F_{n}$ is principally non-singular if both conditions below are satisfied:
\begin{enumerate}[leftmargin=*, label=(\roman*)]
\item The ideal $\bra{1-\zeta_n}$ is prime in $\mathbb Z[\zeta_{mn}]$, and the characteristic of $\mathbb Z[\zeta_{mn}]/\bra{1-\zeta_n}$ is either $0$ or at least $n$;  \label{Cond1}
\item $f(F)$ is principally non-singular in $\mathbb Z[\zeta_{mn}]/\bra{1-\zeta_n}$, where $f$ is the canonical epimorphism from $\mathbb Z[\zeta_{mn}]$ to $\mathbb Z[\zeta_{mn}]/\bra{1-\zeta_n}$.  \label{Cond2}
\end{enumerate}
\end{Thm}

\begin{proof}
Let $G$ be the index set of rows and columns in $F$, we will use $G\times \mathbb Z_n$ to index entries in $F\otimes F_n$. The purpose is to show that $(F\otimes F_n)[\mathcal K, \mathcal K]$ is invertible for any given index subset $\mathcal K\subseteq G\times \mathbb Z_n$. First we shall align some notations:

Set
$$\mathcal I=\{i: \left((i,i'),(j,j')\right)\in\mathcal K\times\mathcal K \}.$$
Intuitively $i\in\mathcal I$ means that $(F\otimes F_n)[\mathcal K, \mathcal K]$ contains some rows from the block 
$$\begin{pmatrix}F[i,0]F_n, & \ldots, & F[i,|G|-1]F_n\end{pmatrix}.$$
For each fixed $i\in G$, further set
$$\mathcal I_i=\{i': \left((i,i'),(j,j')\right)\in\mathcal K\times\mathcal K\},$$
so that $\mathcal I_i$ contains row indices from the above block that are active in $(F\otimes F_n)[\mathcal K, \mathcal K]$. 

Similarly we will use 
$$\mathcal J=\{j: \left((i,i'),(j,j')\right)\in\mathcal K\times\mathcal K\}, \quad \mathcal J_j=\{j': \left((i,i'),(j,j')\right)\in\mathcal K\times\mathcal K\},$$
respectively for active vertical blocks and active column indices in corresponding vertical blocks. Since $(F\otimes F_n)[\mathcal K, \mathcal K]$ is a principal submatrix we certainly have $\mathcal I=\mathcal J$, and $\mathcal I_i=\mathcal J_j$ if $i=j$. 

Assume the contrary that $F\otimes F_{n}$ is principally singular instead, then there is some $\mathcal K$ such that $(F\otimes F_n)[\mathcal K, \mathcal K]$ is not invertible. We may view $(F\otimes F_n)[\mathcal K, \mathcal K]$ as a linear operator on $\left(\mathbb Q(\zeta_{mn})\right)^{|\mathcal K|}$, then there is some non-zero $\vec z\in \left(\mathbb Q(\zeta_{mn})\right)^{|\mathcal K|}$ so that 
$$(F\otimes F_n)[\mathcal K, \mathcal K]\vec z=0.$$ 
Clearly by scaling we can take $\vec z\in\left(\mathbb Z[\zeta_{mn}]\right)^{|\mathcal K|}$ instead. Moreover, if $\vec z$ is not identically zero, then we may without loss of generality assume that there is at least one entry in it that is not divisible by $1-\zeta_n$, otherwise we may simply factor $1-\zeta_n$ out from $\vec z$. 

For each $j\in\mathcal J$, let $\vec z^{(j)}\in\left(\mathbb Z(\zeta_{mn})\right)^{|\mathcal J_j|}$ be the subvector of $\vec z$ obtained by keeping only entries whose indices are in $\mathcal J_j$, then
$$(F\otimes F_n)[\mathcal K, \mathcal K]\vec z=0 \quad \Leftrightarrow \quad \sum_{j\in\mathcal J}F[i,j]F_n\vec z^{(j)}=0, \forall i\in\mathcal I.$$ 
Since $F_n$ is a Vandermonde matrix, for each $j\in\mathcal J$ we may further define the polynomial $S_j(x)\in \mathbb Z[\zeta_{mn}][x]$ to be
$$S_j(x)=\sum_{k\in\mathcal J_j}z^{(j)}_kx^k,$$
where $z^{(j)}_k$ is the entry indexed by $k$ in $z^{(j)}$, then since $\mathcal J_j\subseteq\mathbb Z_n$ by construction, each $S_j(x)$ is a polynomial of order at most $n-1$, and we get
$$\sum_{j\in\mathcal J}F[i,j]F_n\vec z^{(j)}=0, \forall i\in\mathcal I \quad \Leftrightarrow \quad \sum_{j\in\mathcal J}F[i,j]S_j(\zeta_n^{\ell})=0, \forall i\in\mathcal I, \forall \ell\in\mathcal I_i.$$
Now for each $i\in\mathcal I$, define the polynomial $T_i(x)\in \mathbb Z[\zeta_{mn}][x]$ to be
$$T_i(x)=\sum_{j\in\mathcal J}F[i,j]S_j(x),$$
and let $\vec T(x)$ and $\vec S(x)$ be vectors formed by $T_i(x)$ and $S_j(x)$ respectively, then (recall that $\mathcal I=\mathcal J$) we obtain
$$\vec T(x)=F[\mathcal I, \mathcal J]\vec S(x)=F[\mathcal I, \mathcal I]\vec S(x),$$
and each $T_i(x)$ vanishes on $\zeta_n^{\ell}$ for all $\ell\in\mathcal I_i$.   

Set $D=\mathbb Z[\zeta_{mn}]/\bra{1-\zeta_n}$, then \ref{Cond1} implies that $D$ is an integral domain. Apply the canonical epimorphism $f: \mathbb Z[\zeta_{mn}]\to D$ to coefficients of each $T_i(x)$, we obtain a new polynomial $\tilde T_i(x)\in D[x]$. Since $\zeta_n=-(1-\zeta_n)+1$ we see that $f(\zeta_n)=1$, therefore $f(\zeta_n^k)=f(\zeta_n)^k=1$ for all $k\in\mathbb Z_n$. Now $T_i(x)$ vanishing on $\zeta_n^{\ell}$ for all $\ell\in\mathcal I_i$ implies that it is divisible $\prod_{\ell\in\mathcal I_i}(x-\zeta_n^{\ell})$, consequently $\tilde T_i(x)$ is divisible by $(x-1)^{|\mathcal I_i|}$. 

Let $\tilde S_j(x), \tilde S(x), \tilde T(x), \tilde F$ be the images of $S_j(x), \vec S(x), \vec T(x), F$ under $f$ respectively. By \ref{Cond2} we also have that
$$\tilde S(x)=(\tilde F[\mathcal I, \mathcal I])^{-1}\tilde T(x).$$
For each $i\in\mathcal I$ and each $j\in\mathcal J$, denote the entry $(\tilde F[\mathcal I, \mathcal I])^{-1}[j,i]$ (this is well defined since $\mathcal J=\mathcal I$) by $\tilde c_{j,i}$, then we have
$$\tilde S_j(x)=\sum_{i\in\mathcal I}\tilde c_{j,i}\tilde T_i(x). $$
Now let $i_0, i_1,\ldots i_{|\mathcal I|-1}$ be an enumeration of the set $\mathcal I$ with the property that 
$$|\mathcal I_{i_0}|\le |\mathcal I_{i_1}|\le \ldots \le |\mathcal I_{i_{|\mathcal I|-1}}|,$$ 
and set $j_0=i_0, j_1=i_1, \ldots, j_{|\mathcal J|-1}=i_{|\mathcal I|-1}$ (This is well defined since $\mathcal I=\mathcal J$). Then the polynomial $\tilde S_{j_0}(x)$ has $1$ as a root of multiplicity $|I_{i_0}|$, but by construction it has at most $|J_{j_0}|=|I_{i_0}|$ non-zero coefficients, therefore by Lemma \ref{LmF1} this indicates that $\tilde S_{j_0}(x)$ is identically $0$. Repeat the procedure on $\mathcal I\setminus\{i_0\}$, and $\mathcal J\setminus\{j_0\}$, we assert that $\tilde S_{j_1}(x)$ is identically $0$, and so forth $\tilde S_{j_2}(x), \ldots \tilde S_{j_{|\mathcal J|-1}}(x)$ are all identically $0$. This means for every every $j\in\mathcal J$ and every $k\in\mathcal J_j$, the term $z^{(j)}_k$ is divisible by $1-\zeta_n$, i.e., every entry of $\vec z$ is divisible by $1-\zeta_n$. This contradicts the assumption that at least one entry of $\vec z$ shall be not divisible by $1-\zeta_n$. Therefore such $\mathcal K$ does not exist, and $F\otimes F_n$ has to be principally non-singular.
\end{proof}

Assumption \ref{Cond1} of Theorem \ref{ThmMain} ensures that reduction modulo $1-\zeta_n$ yields an integral domain, allowing us to apply Lemma \ref{LmF1}, and carry out multiplicity arguments without zero divisors. By Lemma 2 below, this condition is also equivalent to $n$ being a primitive root modulo $m$.

Let $\varphi$ be the Euler totient function, and $\mathbb F_{p^k}$ the finite field of $p^k$ elements. It is well known (see e.g., \cite[Corollary 2.1.25]{cohen2007}) that $\mathbb Z_m^*$ is cyclic if and only if $m\in\{1,2,4,p^k, 2p^k: p \text{ odd prime}, k\in\mathbb N\}$. 

\begin{Lm}\cite[Lemma 3]{loukaki2025} \label{LmL1}
Let $q$ be an odd prime and $m\in\mathbb N$ with $\gcd(m,q)=1$, then the ideal $\bra{1-\zeta_q}$ is prime in $\mathbb Z[\zeta_{mq}]$ if and only if $q$ generates $\mathbb Z_m^*$ (hence $\mathbb Z_m^*$  is cyclic). Moreover, in such case 
$$\mathbb Z[\zeta_{mq}]/\bra{1-\zeta_q}\cong \mathbb F_{q^{\varphi(m)}}.$$ 
\end{Lm}

\begin{Lm}\cite[Lemma 4]{loukaki2025} \label{LmL2}
Let $p,q$ be distinct odd primes, then the image of $\zeta_p$ under the canonical epimorphism from $\mathbb Z[\zeta_{pq}]$ to $\mathbb Z[\zeta_{pq}]/\bra{1-\zeta^q}$ is still a primitive $p$-th root of unity.
\end{Lm}

\begin{Lm}\cite[Theorem A]{zhang2019} \label{LmZ}
For every odd prime $p$, there is a constant $\Gamma_p$ such that if $q>\Gamma_p$ is another odd prime that generates $\mathbb Z_p^*$, then all minors of $F_p$ (interpreted as the Vandermonde matrix on $p$-th roots of unity in $\mathbb F_{q^{p-1}}$) are non-vanishing in $\mathbb F_{q^{p-1}}$.
\end{Lm}

Remark: Roots of unity behave differently in finite fields compared to in extensions of $\mathbb Q$, for example $2$ is an $11$-th primitive root of unity in $\mathbb F_{23}$, but $2^4+2^2+2^1+2^0=23=0 \pmod {23}$, so it vanishes on a polynomial of order $4$ instead of order $10$. Therefore in contrast to what the Chebotar\"ev theorem asserted in $\mathbb C$, it is actually possible for $F_p$ to have vanishing minors in finite fields (see \cite[Example 4.1, Example 4.3]{zhang2019}). Moreover, in the setting of Lemma \ref{LmZ}, the $p$-th cyclotomic polynomial will actually decompose into linear factors in $\mathbb F_{q^{p-1}}$ (this can be derived by e.g., applying \cite[Theorem 59]{waldschmidt2014}) since $q^{p-1}=1 \pmod p$, which is why the statement of Lemma \ref{LmZ} is non-trivial. The proof in \cite[Theorem A]{zhang2019}, as well as methods in e.g., \cite{delvaux2008}, use the formula of generalized Vandermonde determinant in \cite{evans1977} and checks the sum of coefficients of the Schur function in that formula. 

\begin{Cor} \label{CorPQ}
Let $p,q$ be distinct odd primes, if $q$ generates $\mathbb Z_{p}^*$ and $q>\Gamma_p$ (whose value is defined in Lemma \ref{LmZ}), then $F_p\otimes F_q$ is principally non-singular.
\end{Cor}

\begin{proof}
The statement is proved by applying Theorem \ref{ThmMain} on $F=F_p$ and $n=q$. Indeed, when we apply the canonical epimorphism $f: \mathbb Z[\zeta_{pq}]\to \mathbb Z[\zeta_{pq}]/\bra{1-\zeta_q}$, we see that \ref{Cond1} is satisfied by Lemma \ref{LmL1}, while \ref{Cond2} is satisfied by combining Lemma \ref{LmL2} (which indicates that $f(\zeta_p)$ remains a primitive $p$-th root of unity) and Lemma \ref{LmZ} (which shows that $f(F_p)$ is principally non-singular). Therefore by Theorem \ref{ThmMain}, $F_p\otimes F_q$ is principally non-singular. 
\end{proof}

\begin{Lm}\cite[Lemma 1]{frenkel2003} \label{LmF2}
If $q$ is a prime, then the ideal $\bra{1-\zeta_q}$ is prime in $\mathbb Z[\zeta_q]$, and $\mathbb Z[\zeta_q]/\bra{1-\zeta_q}\cong \mathbb F_q$.
\end{Lm}

\begin{Cor} \label{Cor2Q}
Let $F_2^{\otimes k}=\underbrace{F_2\otimes\ldots\otimes F_2}_{k \text{ terms}}$, if $q$ is an odd prime and $k\in\{1,2,3\}$, then $F_2^{\otimes k}\otimes F_q$ is principally non-singular upon permutation. In particular, for $k=1$ the identity permutation suffices.
\end{Cor}

\begin{proof}
We will prove the statement by applying Theorem \ref{ThmMain} with $F=F_2^{\otimes k}P$ and $n=q$ for some appropriate permutation matrix $P$. Notice also that in this setting we have $m=2$.

Since $q$ is odd we get $\zeta_{2q}=-\zeta_q^2$, which indicates that $\mathbb Z(\zeta_{2q})=\mathbb Z(\zeta_{q})$. Thus applying Lemma \ref{LmF2} we see that \ref{Cond1} holds for all $k\in\mathbb N$.

Every entry in $F_2^{\otimes k}$ is either $1$ or $-1$, which is still $\pm 1$ under the canonical epimorphism from $\mathbb Z[\zeta_q]$ to $\mathbb Z[\zeta_q]/\bra{1-\zeta_q}$. A principal minor will only vanish if it is divisible by $q$. Therefore if there is a permutation $P$ such that every principal minor in $F_2^{\otimes k}P$ is a power of $2$, then \ref{Cond2} will be true for all $q$.
 
If $k=1$, then this trivially holds with $P$ being the identity permutation $I$. For $k=2,3$ the corresponding permutations and the principal non-singularities after permutations are produced and verified by an exhaustive search using the Python program in the appendix. There are $8$ such permutations for $k=2$ and $64$ such permutations for $k=3$.

Now as both conditions are satisfied, we can conclude by Theorem \ref{ThmMain} that $F\otimes F_q$ is principally non-singular. Then using the identity (see e.g., \cite[Lemma 4.2.10]{horn1991}) that $(A\otimes B)(C\otimes D)=(AC)\otimes (BD)$ when the multiplications $AC,BD$ are well defined, and denote by $I$ the identity permutation on $\mathbb C^{q\times q}$ we get
$$F\otimes F_q=(F_2^{\otimes k}P)\otimes (F_qI)=(F_2^{\otimes k}\otimes F_q)(P\otimes I).$$
Since $P\otimes I$ is again a permutation matrix, this indicates that $F_2^{\otimes k}\otimes F_q$ is principally non-singular upon permutation.
\end{proof}

It is worth mentioning that $F_2^{\otimes k}$ is a real Hadamard matrix (of Sylvester type), and therefore a representation of the $k$-th order Walsh system. Walsh functions are piecewise constant functions on the unit interval that takes value $\pm 1$, they form an orthonormal basis of $L^2([0,1])$. The $k$-th order Walsh system is a subset of Walsh functions consists of $2^k$ mutually orthogonal elements, each of which is constant on the intervals $[j/2^k, (j+1)/2^k)$ with $j=0,1,\ldots,2^k-1$. By a proper ordering they correspond to a real Hadamard matrix with the $i,j$-th entry of the matrix being the value of the $i$-th function on the interval $[j/2^k, (j+1)/2^k)$. The principal non-singularity upon permutation indicates that we can restrict these functions to any subcollections of dyadic intervals, and there will still be a subsystem that remains linearly independent. It is also tempting to ask whether the principal non-singularity upon permutation in all finite fields is a property also possessed by other real Hadamard matrices.

\section*{Appendix}
The following Python 3 code (prepared with Python 3.9.2, Scipy 1.6.0, Numpy 1.19.5. On a Raspberry Pi 400 microcomputer it runs for about 10 seconds) is used to verify that $F_2^{\otimes 2}$ and $F_2^{\otimes 3}$ are principally non-singular upon permutation in finite fields. It merely generates all possible column permutations and for each permutation looks through all principal minors. The largest principal submatrices are those $7\times 7$ submatrices in $F_2^{\otimes 3}$, their determinants will not exceed $7!=5040<8192=2^{13}$ in absolute value (indeed, the determinant of an $n\times n$ matrix is a multivariate polynomial consists of precisely $n!$ terms, each term is a product of matrix entries. Since all entries in the matrices here are $\pm 1$, the determinant must be bounded by $n!$ in absolute value), thus for each principal minor we only need to check whether its absolute value is in the set $\{2^k: k=0,1,\ldots,13\}$.

\begin{Verbatim}[frame=lines]
import scipy.linalg
import itertools,numpy,math

bound=3

bPowers={1}
x=1
while x<=math.factorial((1<<bound)-1):
    x*=2
    bPowers.add(x)

F2=((1,1),(1,-1))
F=[[1,1],[1,-1]]
n=2
for k in range(2,bound+1):
    F=numpy.kron(F,F2)
    n*=2
    res=[]
    for perm in itertools.permutations(range(n)):
        isSingular=False
        for m in range(2,n):
            for comb in itertools.combinations(range(n),m):
                I,J=numpy.ix_(list(comb),[perm[x] for x in comb])
                d=round(abs(scipy.linalg.det(F[I,J])))
                if d not in bPowers:
                    isSingular=True
                    break
            if isSingular:
                break
        if not isSingular:
            res.append(perm)
    print(f"There are {len(res)} valid permutations for k={k}:")
    *map(print,res),
         
print("Job Completed.")
\end{Verbatim}

\end{document}